\documentclass[12pt]{article}
\usepackage{latexsym,amsmath,amssymb}

\newif\ifdviwin

\usepackage{enumerate}
\usepackage[latin1]{inputenc}
\usepackage[english,spanish]{babel}
\usepackage{indentfirst}
\usepackage[mathscr]{eucal}
\usepackage{amssymb,amsmath,amsfonts}
\usepackage{graphicx}

\newif\ifdviwin

\dviwintrue

\def\vart{\vartheta}
\let\8=\infty \let\0=\emptyset

\let\tilde=\widetilde

\let\landa=\lambda
\let\alfa=\alpha

\def\ep{\varepsilon}
\def\landa{\lambda}

\def\flecha{\rightarrow}
\def\esiz{\langle}
\def\esde{\rangle}

\def\cte.{\mathop{\rm cte.}\nolimits}

\def\det{\mathop{\rm det}\nolimits}

\def\Re{\mathop{\rm Re }\nolimits}

\def\N{\mathbb{N}}
\def\L{\mathbb{L}}

\def\R{\mathbb{R}}

\def\C{\mathbb{C}}
\def\D{\mathbb{D}}

\def\H{\mathbb{H}}
\def\S{\mathbb{S}}
\def\X{\mathfrak{X}}

\def\Nil{{\rm{Nil}_3}}
\def\Ek{\mathbb{E}^3 (\kappa,\tau)}
\def\Mk{\mathcal{M}^2(\kappa)}

 \newtheorem{defi}{Definition}
 \newtheorem{teo}[defi]{Theorem}
 \newtheorem{pro}[defi]{Proposition}
 \newtheorem{cor}[defi]{Corollary}
 \newtheorem{lem}[defi]{Lemma}
 
 \newtheorem{remark}[defi]{Remark}

 \newenvironment{proof}{\rm \trivlist \item[\hskip \labelsep{\it
      Proof}:]}{\par\nopagebreak \hfill $\Box$ \endtrivlist}

\newcommand\esc[2]{\langle{#1},{#2}\rangle}

\numberwithin{equation}{section}

\begin{document}
\mbox{}\vspace{0.4cm}\mbox{}

\begin{center}
\rule{14cm}{1.5pt}\vspace{0.5cm}

{\Large \bf Holomorphic quadratic differentials and } \\ [0.3cm]{\Large \bf the Bernstein problem in Heisenberg space}\\
\vspace{0.5cm} {\large Isabel Fernández$\mbox{}^a$ and Pablo
Mira$\mbox{}^b$}\\ \vspace{0.3cm} \rule{14cm}{1.5pt}
\end{center}
  \vspace{1cm}
$\mbox{}^a$ Departamento de Matemáticas, Universidad de
Extremadura, E-06071 Badajoz, Spain. \\ e-mail: isafer@ugr.es
\vspace{0.2cm}

\noindent $\mbox{}^b$ Departamento de Matemática Aplicada y
Estadística,
Universidad Politécnica de Cartagena, E-30203 Cartagena, Murcia, Spain. \\
e-mail: pablo.mira@upct.es \vspace{0.2cm}

\vspace{0.2cm}



\vspace{0.3cm}

 \begin{abstract}
We classify the entire minimal vertical graphs in the Heisenberg
group ${\rm Nil_3}$ endowed with a Riemannian left-invariant
metric. This classification, which provides a solution to the
Bernstein problem in ${\rm Nil}_3$, is given in terms of the
Abresch-Rosenberg holomorphic differential for minimal surfaces
in ${\rm Nil}_3$.
 \end{abstract}

\section{Introduction}\label{sec:intro}

A recent major achievement in the theory of constant mean
curvature (CMC) surfaces in Riemannian $3$-manifolds has been the
discovery by Abresch and Rosenberg \cite{AbRo1,AbRo2} of a
holomorphic quadratic differential for every CMC surface in any
homogeneous $3$-manifold with a $4$-dimensional isometry group.
The most remarkable consequence of this construction is the
classification given in \cite{AbRo1,AbRo2} of the immersed CMC
spheres in these homogeneous spaces as embedded rotational ones,
thus solving an open problem proposed in \cite{HsHs}.

Our purpose in this paper is to solve by means of the
Abresch-Rosenberg differential another popular open problem in
the theory of CMC surfaces in homogeneous spaces. Namely, the
Bernstein problem in the Riemannian Heisenberg $3$-space $\Nil$,
i.e. the classification of the entire minimal vertical graphs in
$\Nil$.

Let us fix some terminology. By using exponential coordinates and
up to re-scaling, we shall regard the \emph{Heisenberg space}
$\Nil$ as $\R^3$ endowed with the Riemannian metric $$ d\sigma^2
:= dx^2 +dy^2 + \left(dz + \frac{1}{2}( y dx -x dy)\right)^2.$$
This Riemannian metric is left-invariant for the Lie group
structure of Heisenberg space. In this way $\Nil$ becomes one of
the five simply-connected canonical homogeneous $3$-manifolds
with a $4$-dimensional isometry group.

A (vertical) graph $z= f(x,y)$ in $\Nil$ is said to be minimal if
its mean curvature vanishes identically, and \emph{entire} if it
is defined for every $(x,y)\in \R^2$. The \emph{Bernstein
problem} in $\Nil$ asks for the classification of all entire
vertical minimal graphs of $\Nil$. Equivalently, we wish to find
all the $C^2$ solutions globally defined on $\R^2$ of the
quasilinear elliptic PDE
 \begin{equation}\label{min}
(1+ \beta^2) f_{xx}  - 2 \alfa \beta \, f_{xy} +
(1+\alfa^2)f_{yy} =0,
 \end{equation}
where $\alfa:= f_x + y/2$ and $\beta := f_y - x/2$ (see
\cite{ADR,Iea} for instance).

The simplest entire solutions to \eqref{min} are the linear ones,
which give rise to entire minimal graphs in $\Nil$ that are
called \emph{horizontal umbrellas} \cite{AbRo2}. In this model
they are just non-vertical planes, but they are not totally
geodesic. Other simple solution is $f(x,y)= xy/2$, and more
generally, the $1$-parameter family \cite{FMP}
 \begin{equation}\label{saddle}
 f(x,y)= \frac{xy}{2} - c \left( y\sqrt{1+y^2} + \log \left( y + \sqrt{1+y^2}\right)\right), \hspace{0.8cm} c\in \R.
 \end{equation}
The resulting entire minimal graphs are called \emph{saddle-type
minimal examples} \cite{AbRo2}. There are other explicit entire
solutions to \eqref{min} foliated by Euclidean straight lines (see
\cite{Dan2}), what suggests that the class of entire minimal
graphs in $\Nil$ is quite large.

The Bernstein problem in $\Nil$ is a usual discussion topic among
people working on CMC surfaces in homogeneous spaces, and appears
explicitly as an open problem in several works like
\cite{FMP,AbRo2,Dan2,Iea,MMP}. Let us also point out that the
Bernstein problem in the sub-Riemannian Heisenberg space $\H^1$
has been intensely studied in the last few years (see
\cite{BSV,CHMY,GaPa,RiRo} and references therein).

In this work we show that the moduli space of entire minimal
graphs in $\Nil$ can be parametrized in terms of their
Abresch-Rosenberg differentials. Specifically, we prove:

\begin{teo}[Main Theorem]\label{th:main}
Let $Q dz^2$ be a holomorphic quadratic differential in $\Sigma
\equiv \C$ or $\D$ which is not identically zero if $\Sigma\equiv
\C$. Then there exists a $2$-parameter family of (generically
non-congruent) entire minimal graphs in $\Nil$ whose
Abresch-Rosenberg differential is $Qdz^2$.

Conversely, any entire minimal graph in $\Nil$ is one of the
above examples.
\end{teo}

It is worth mentioning that if two holomorphic quadratic
differentials on $\C$ or $\D$ only differ by a global conformal
change of parameter, then the families of minimal graphs that they
induce are the same (the corresponding immersions differ by this
conformal global change of parameter).

We also remark that the conformally immersed minimal surfaces
$X:\Sigma\flecha \Nil$ with $\Sigma\equiv \C$ and
Abresch-Rosenberg differential $Q dz^2 = c\, dz^2$, $c\neq 0$,
are exactly the saddle-type minimal examples \eqref{saddle}, see
Section \ref{sec:prelim}. In that case some of the elements of
the $2$-parameter family are mutually congruent. On the other
hand, the minimal surfaces in $\Nil$ with vanishing
Abresch-Rosenberg differential on $\Sigma=\C$ are the vertical
planes, which are not entire graphs.

The solution to the Bernstein problem exposed in Theorem
\ref{th:main} is sharp, in the following sense. First, it tells
that the space of entire solutions to \eqref{min} is extremely
large. So, as these solutions will not be explicit in general,
the most reasonable answer to the Bernstein problem in $\Nil$
that one can look for is a description of the moduli space of
entire minimal graphs in terms of some already known class of
objects. With this, Theorem \ref{th:main} tells that the space of
entire minimal graphs in $\Nil$ can be described by the space of
holomorphic quadratic differentials on $\C$ or $\D$, which is a
well known class.

The proof of Theorem \ref{th:main} is based on four pillars.
First, the Abresch-Rosenberg holomorphic differential, whose role
is substantial since it is the classifying object in the solution
to the Bernstein problem. Second, the construction by the authors
in \cite{FeMi1} of a harmonic Gauss map into the hyperbolic plane
$\H^2$ for CMC surfaces with $H=1/2$ in $\H^2\times \R$, and the
global implications of this construction. Actually, the present
paper may be regarded as a natural continuation of our previous
work \cite{FeMi1} in many aspects. The third pillar is the
Lawson-type isometric correspondence by Daniel \cite{Dan1}
between minimal surfaces in $\Nil$ and $H=1/2$ surfaces in
$\H^2\times \R$. And fourthly, we will use some deep results by
Wan-Au \cite{Wan,WaAu} on the (at a first sight unrelated)
classification of the complete spacelike CMC surfaces in the
Lorentz-Minkowski $3$-space $\L^3$.

We have organized the paper as follows. Section \ref{sec:prelim}
will be devoted to analyze in detail the behavior of a large
family of complete minimal surfaces constructed in \cite{FeMi1}
that are local graphs in $\Nil$. We shall call them
\emph{canonical examples} in $\Nil$. For that, we will revise all
the basic machinery of the Abresch-Rosenberg differential, the
Lawson-type correspondence in \cite{Dan1} and the harmonic Gauss
map for $H=1/2$ surfaces in $\H^2\times \R$.

In Section \ref{mainsec} we will prove that these canonical
examples are precisely the entire minimal graphs in $\Nil$. This
result and the constructions of Section \ref{sec:prelim} imply
Theorem \ref{th:main}, and thus solve the Bernstein problem in
$\Nil$.

Finally, in Section \ref{sec:H2xR} we shall give a large family
of entire vertical graphs with $H=1/2$ in $\H^2\times \R$, what
can be seen as an advance in the problem of classifying all such
entire graphs.

\subsection*{Acknowledgements}

The authors wish to express their gratitude to the referee of
this paper for his valuable comments and suggestions.

Isabel Fernández was partially supported by MEC-FEDER Grant No.
MTM2004-00160. Pablo Mira was partially supported by MEC-FEDER,
Grant No. MTM2007-65249.

\section{Setup: the canonical examples in ${\rm Nil}_3$}\label{sec:prelim}

In this section we introduce the Heisenberg space $\Nil$ and the
product space $\H^2\times\R$ as Riemannian homogeneous $3$-spaces
with $4$-dimensional group of isometries and we describe some
basic facts about the theory of surfaces in these spaces that will
be useful for our purposes. We shall also analyze the
\emph{canonical examples}, which will be characterized in Section
3 as the only entire minimal graphs in $\Nil$. In order to do
that, we will recall the correspondence between minimal surfaces
in $\Nil$ and CMC $1/2$ surfaces in $\H^2\times\R$, as well as
the notion of {\em hyperbolic Gauss map} for surfaces in
$\H^2\times\R$ and its basic properties. For more details about
these matters we refer to \cite{AbRo2,Dan1,Dan2,FeMi1,FeMi2} and
references therein.

%%%%%%%%%%%%%%%%%%%%%%%%
\subsection*{Surfaces in homogeneous $3$-spaces}\label{sub:homog}

We will label as $\Mk$ the unique $2$-dimensional Riemannian
space form of constant curvature $\kappa$. Thus, $\Mk=\R^2, \S^2$
or $\H^2$ when  $\kappa=0,1$ or $-1$, respectively.

Any simply-connected homogeneous $3$-manifold with a
$4$-dimensional group of isometries can be described as a
Riemannian fibration over some $\Mk$. Translations along the
fibers are isometries and therefore they generate a Killing
field, $\xi,$ also called {\em the vertical field}. The
\emph{bundle curvature} of the space is the number $\tau\in\R$
with $\kappa\neq 4\tau^2$ such that $ \overline{\nabla}_X
\xi=\tau X\times\xi $ holds for any vector field $X$ on the
manifold. Here $\overline{\nabla}$ is the Levi-Civita connection
of the manifold and $\times$ denotes the cross product. When
$\tau=0$ this fibration becomes trivial and thus we get the
product spaces $\Mk\times\mathbb{R}$. When $\tau\neq 0$ the
manifolds have the isometry group of the Heisenberg space if
$\kappa=0,$ of the Berger spheres if $\kappa>0$, or the one of
the universal covering of $\mbox{PSL}(2,\mathbb{R})$ when
$\kappa<0$.

In what follows $\mathbb{E}^3(\kappa,\tau)$ will represent a
simply connected homogeneous 3-manifold with isometry group of
dimension 4, where $\kappa$ and $\tau$ are the real numbers
described above.

For an immersion $X:\Sigma\to\Ek$, we will define its
\emph{vertical projection} as $F:=\pi\circ\psi:\Sigma\to\Mk$,
where $\pi :\Ek\flecha \Mk$ is the canonical submersion of the
homogeneous space. A surface in $\Ek$ is a (vertical) entire
graph when its vertical projection is a global diffeomorphism.
From now on, by a \emph{graph} we will always mean a vertical one.

\begin{remark}\label{re:heis}
The case we are mostly interested in is, up to re-scaling, when
$\kappa=0$ and $\tau=1/2$. This corresponds to the Heisenberg
space, that will denoted by $\Nil$. This space can be seen as
${\Nil}=(\R^3,d\sigma^2)$ where $d\sigma^2$ is the Riemannian
metric $$ d\sigma^2 := dx_1^2 +dx_2^2 + \left(dx_3 + \frac{1}{2}(
x_2 dx_1 -x_1 dx_2)\right)^2.$$ In this model the Riemannian
submersion $\pi$ is just the usual projection of $\R^3$ onto the
plane $x_3=0$, $$\pi:{\Nil}\to\R^2, \quad
\pi(x_1,x_2,x_3)=(x_1,x_2).$$
\end{remark}

Let $X:\Sigma\to\Ek$ be an immersed surface with unit normal
vector $\eta$. Given a conformal parameter $z$ for $X$ let us
write its induced metric as $\landa|dz|^2$. We also label $H$ as
the mean curvature of the surface and $p\,dz^2$ as its Hopf
differential, i.e. $p=-\esc{X_z}{\eta_z}$. The {\em angle
function} is defined as the normal component of the vertical
field, $u=\esc{\eta}{\xi}$. Finally, we define the complex-valued
function $A$ as $A=\esc{\xi}{X_z}=\esc{T}{\partial_z},$ where
$T\in\mathfrak{X}(\Sigma)$ is given by $dX(T)=\xi-u\eta.$ We will
call  $\{\landa,u,H,p,A\}\in\R^+\times [-1,1]\times\R\times \C^2$
the {\em fundamental data } of the surface $X:\Sigma\to\Ek$.

The integrability conditions of surfaces in $\Ek$ can be
expressed in terms of some equations involving these fundamental
data (see \cite{Dan1,FeMi2}). Of these, in this paper we will
only need the relation $$ {\bf (C.4)} \hspace{1cm} \frac{4
|A|^2}{\landa}  =  1 - u^2 .$$

With the above notations, if $F:\Sigma\to\Mk$ is the vertical
projection of $X:\Sigma\to\Ek$, then
 \begin{equation}\label{xip}
 F_z=X_z - A\xi.
 \end{equation}
Notice also that when $\tau=0$ (i.e., $\Ek=\Mk\times\R$) the
vertical field $\xi$ is nothing but
$\xi=(0,1)\in\mathfrak{X}(\Mk)\times\R$. Therefore, if we write
$X=(N,h):\Sigma\to\Mk\times\R$ then the vertical projection is
$N$ and $A=h_z$. The function $h$ will be called the {\em height
function}.

\subsection*{The Abresch-Rosenberg differential}

The Abresch-Rosenberg differential for an immersed surface
$X:\Sigma\flecha \Ek$ is defined as the quadratic differential
$$Q dz^2 = \left( 2(H+i\tau) p - (\kappa -4 \tau^2) A^2 \right)
dz^2.$$ Its importance is given by the following result.
\begin{teo}[\cite{AbRo1,AbRo2}]
The Abresch-Rosenberg differential $Q dz^2$ is a holomorphic
quadratic differential on any CMC surface in $\Ek$.
\end{teo}
Unlike the case of $\R^3$, where the holomorphicity of the Hopf
differential necessarily implies the constancy of the mean
curvature, there exist non-CMC surfaces in some homogeneous
spaces $\Ek$ with holomorphic Abresch-Rosenberg differential. A
detailed discussion of this converse problem can be found in
\cite{FeMi2}.

In \cite{Dan1} B. Daniel established a local isometric Lawson type
correspondence between simply connected CMC surfaces in the
homogeneous spaces $\Ek$. Surfaces related by this correspondence
are usually called {\em sister surfaces}.

As a particular case this isometric correspondence relates
minimal surfaces in $\Nil$ and $H=1/2$ surfaces in
$\H^2\times\R$. In terms of their fundamental data the
correspondence is given as follows: given a simply connected
minimal surface in $\Nil$ with fundamental data
\begin{equation}\label{eq:fund1}
\{ \landa,u,H=0,p,A\} \end{equation}
 its sister $H=1/2$ surface in
$\H^2\times\R$ (which is unique up to isometries of $\H^2\times
\R$ preserving the orientation of both base and fiber) has
fundamental data
\begin{equation}\label{eq:fund2}
\{ \landa, u, H=1/2, -ip,-iA \}. \end{equation} In particular,
their respective Abresch-Rosenberg differentials are opposite to
one another. Moreover, the two sister surfaces have the same
induced metric, and one of them is a local graph if and only if
its sister immersion is also a local graph (if and only if $u\neq
0$). We remark that in this paper we shall prove that the sister
surface of an entire minimal graph in $\Nil$ is always an entire
$H=1/2$ graph in $\H^2\times \R$ (see Proposition \ref{proH2}).

%%%%%%%%%%%%%%%%%%%%%%%%%%%%%%%%%%%%%%%%%%%%%%%%%%%%%%%%%%%%%%%%%%%%%%%%%%%%%%%%

\subsection*{The hyperbolic Gauss map in $\H^2\times \R$}

One of the key tools of this work is the existence of a Gauss-type map for surfaces in $\H^2\times \R$ with special properties, introduced in \cite{FeMi1}. Next, we analyze in more detail than in \cite{FeMi1} the geometrical meaning of this Gauss map.

First, let $\L^4$ denote the Minkowski $4$-space with canonical coordinates $(x_0,x_1,x_2,x_3)$ and the Lorentzian metric $\esiz,\esde = -dx_0^2 + \sum_{i=1}^3 dx_i^2$. We shall regard as hyperquadrics in $\L^4$  in the usual way the hyperbolic $3$-space $\H^3$, the de Sitter $3$-space $\S_1^3$, the positive null cone $\N^3$, and the product space $\H^2\times \R$ (see \cite{FeMi1}).

Let $\S_{\8}^2$ denote the ideal boundary of $\H^3$. We have the usual identifications (see \cite{Bry,GaMi} for instance)  $$\S_{\8}^2 \equiv \N^3/\R_+ \equiv \N^3 \cap \{x_0 =1\} \equiv \{1\} \times \S^2 \equiv \S^2\subset \R^3.$$ Let $\psi= (N,h):\Sigma\flecha \H^2\times \R$ be an immersed surface, and let $\eta:\Sigma\flecha \S_1^3$ denote its unit normal. Then $\eta + N $ takes its values in $\N^3$, and $[\eta + N] \in \N^3/\R_+ \equiv \S_{\8}^2$ is the point at infinity reached by the unique geodesic of $\H^3$ starting at $N $ with initial speed $\eta$.

If $\psi$ has regular vertical projection (i.e. $N:\Sigma\flecha \H^2$ is a local diffeomorphism), and $u$ denotes the last coordinate of $\eta$, it turns out that $u\neq 0$ at every point (we assume thus that $u>0$ up to a change of orientation). Consequently, the map $$\xi = \frac{1}{u} (\eta + N):\Sigma\flecha \L^4$$ takes its values in $\N^3\cap \{x_3=1\}$. Thus, there is some $G:\Sigma\flecha \H^2$ with $\xi = (G,1)$.

Now, observe that as $u>0$, it happens that $[\eta + N] = [\xi]$ takes its values in the hemisphere $$\S_{\8,+}^2 \equiv (\N^3 \cap \{x_3 >0\})/\R_+ \equiv (\{1\}\times \S^2) \cap \{x_3 >0\}\equiv \S^2 \cap \{x_3 >0\}.$$ So, if we write $G=(G_0,G_1,G_2)$ and $[\eta +N] =(1,y_1,y_2,y_3)\equiv (y_1,y_2,y_3)$, we have that, using the projective structure of $\N^3$, $$(y_1,y_2,y_3) = [\eta + N] = [\xi] = \frac{1}{G_0} (G_1,G_2,1) \in \S_+^2:= \S^2\cap \{x_3 >0\}.$$ Composing finally with the vertical projection $\S_+^2 \flecha \D$ we obtain a map $$z\in \Sigma \mapsto (G_0,G_1,G_2)(z) \in \H^2\subset \L^3 \mapsto \left( \frac{G_1}{G_0}, \frac{G_2}{G_0} \right) (z)\in \D.$$ This last coordinate relation is the usual identification between the Minkowski model $\mathcal{H}^2 = \{x\in \L^3 : \esiz x,x\esde = -1, x_0 >0\}$ and the Klein (or projective) model $(\D,ds_K^2)$ of the hyperbolic space $\H^2$.

In other words, $G:\Sigma\flecha \H^2$ can be seen as the map sending each point $(N_0,h_0)$ of the surface to the positive hemisphere at infinity $\S_{\8,+}^2$ of $\H^3$ endowed with the Klein metric of $\H^2$ (after vertical projection), by considering the unique geodesic of $\H^3$ starting at $N_0$ with initial speed $\eta$ (the unit normal at the point).

\begin{defi}
The above map $G:\Sigma\flecha \H^2$ will be called the \emph{hyperbolic Gauss map} of the surface with regular vertical projection $\psi:\Sigma\flecha \H^2\times \R$.
\end{defi}

This interpretation of the hyperbolic Gauss map in terms of the ideal boundary of $\H^3$ enables us to visualize it directly in the cylindrical model of $\H^2\times \R$. This might be useful for a better understanding of the geometrical information carried
by this hyperbolic Gauss map.

So, let us view $\H^2\times \R$ in its cylindrical model $(\D, ds_P^2)\times \R$, where
$$ds^2_P=\frac{4}{(1-|z|^2)^2}|dz|^2, \hspace{1cm} z\in \D = \{z\in \C: |z|<1\},$$ is the Poincaré metric of $\H^2$. We consider as before an immersed surface $\psi:\Sigma\flecha \H^2\times \R$ with regular vertical projection and upwards pointing unit normal. Let $z_0\in\Sigma$ and write $\psi(z_0)=(N_0,h_0)$. Then
we can visualize the hyperbolic $3$-space $\H^3$ (in its Poincaré
ball model) {\em inside} the cylinder $\H^2\times\R$ with equator
at $\H^2\times\{0\}$ (see Figure \ref{fig:gauss}). We now
consider the unique oriented geodesic $\gamma$ in $\H^3$ starting
at $(N_0,0)$ with tangent vector $\eta(z_0)$. Label
$G(z_0)\in\partial_\infty\H^3$ as the limit point of $\gamma$ in
the asymptotic boundary of $\H^3$,
$\partial_\infty\H^3\equiv\S^2$. Let us observe that, since we
have chosen $\eta$ upwards pointing, $G(z_0)$ actually lies in
the upper hemisphere of the sphere at infinity $\S_+^2$. At last, we project $\S_+^2$ into $\D$, and we endow $\D$ with the Klein metric (thus $(\D,ds_K^2)\equiv \H^2$).

\begin{figure}[htbp]\begin{center}
\includegraphics[width=0.45\textwidth]{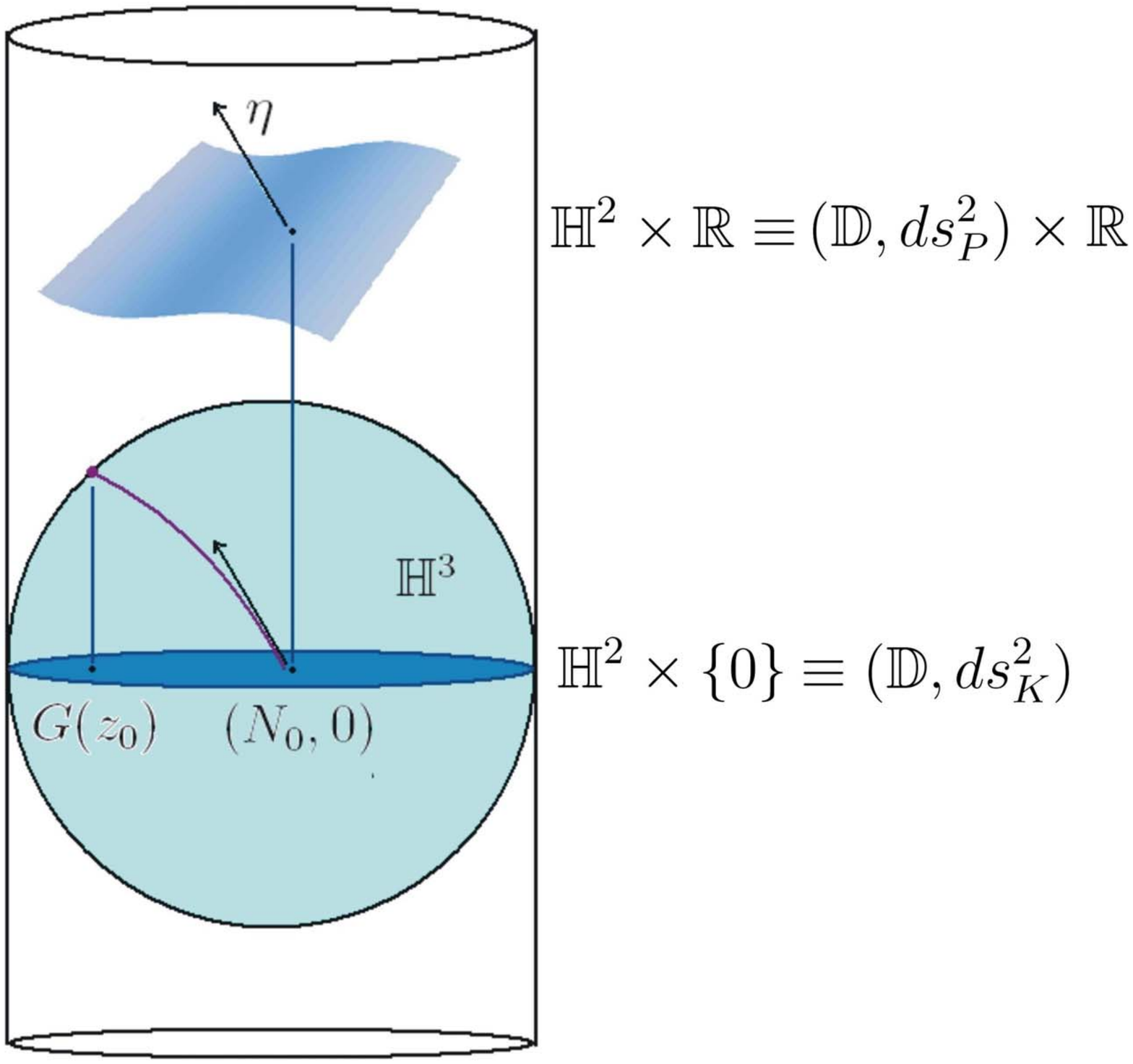}\caption{The hyperbolic Gauss map for surfaces in $\H^2\times\R$.}
\label{fig:gauss}
\end{center}\end{figure}

The result of this is a map $G:\Sigma\to(\D,ds_K^2)\equiv \H^2$ which is the {\em hyperbolic Gauss map} of $\psi$. Its most remarkable property is that it is
harmonic for $H=1/2$ surfaces, as we explain next.

\subsection*{CMC one half surfaces in $\H^2\times\R$}\label{sub:unmedio}

The existence of the {\em canonical examples} of minimal surfaces
in $\Nil$ will be shown in terms of their sister surfaces in
$\H^2\times\R$, so let us first recall some facts about surfaces
with $H=1/2$ in $\H^2\times \R$.

To start, let $G:\Sigma\to\H^2$ denote a harmonic map from the Riemann
surface $\Sigma$ into $\H^2$. Then the quadratic differential
$\esc{G_z}{G_z}dz^2$ is holomorphic on $\Sigma$, and it is called
the \emph{Hopf differential} of the harmonic map $G$. Following
\cite{FeMi1}, we say that a harmonic map {\em admits Weierstrass
data} $\{Q,\tau_0\}$ if there exists $\tau_0:\Sigma\flecha
(0,\8)$ a smooth positive function such that
$$\esc{dG}{dG}=Qdz^2+\left(\frac{\tau_0}{4} +
\frac{4|Q|^2}{\tau_0}\right)|dz|^2+\bar{Q}d\bar{z}^2.$$ The
connection between $H=1/2$ surfaces in $\H^2\times \R$ and
harmonic maps into $\H^2$ was studied in \cite{FeMi1}. We
condense the most relevant features of this connection in the
following result.

\vspace{0.2cm}

\noindent {\bf Convention:} from now on, a surface in $\H^2\times
\R$ with r.v.p. will always be assumed to be \emph{canonically
oriented}, meaning this that its unit normal is upwards pointing.

\begin{teo}[\cite{FeMi1}]\label{th:american}
The hyperbolic Gauss map $G:\Sigma\flecha \H^2$ of a $H=1/2$
surface $\psi:\Sigma\flecha \H^2\times\R$ with regular vertical
projection is harmonic. Moreover, if $-Qdz^2$  denotes the
Abresch-Rosenberg differential of the surface, $\landa|dz|^2$ its
metric and $u$ its angle function, then the quantities
$\{Q,\landa u^2\}$ are Weierstrass data for $G$. In particular
the Hopf differential of $G$ is $Qdz^2$.

Conversely, let $\Sigma$ be a simply-connected Riemann surface and
$G:\Sigma\to\H^2$ be a harmonic map admitting Weierstrass data
$\{Q,\tau_0\}$.Then
\begin{enumerate}
\item If $G$ is singular on an open subset
of $\Sigma$ then it parametrizes a piece of a geodesic in $\H^2$.
Moreover, there exists a $1$-parameter family of $H=1/2$ graphs
in $\H^2\times \R$, invariant under a $1$-parameter subgroup of
isometries, such that $G$ is the hyperbolic Gauss map of any of
them. These surfaces are described in \cite{Sa} (see also
\cite{FeMi1,Dan2}), and they are open pieces of the sister
surfaces of the saddle-type minimal graphs in $\Nil$.
\item If not, let $z_0\in\Sigma$ be a regular point of $G$. Then
for any $\vart_0\in\C$ there exists a unique $H=1/2$ surface
$\psi=(N,h):\Sigma\to\H^2\times\R$ with r.v.p., having $G$ as its
hyperbolic Gauss map and satisfying:
\begin{enumerate}[(a)]
\item $\tau_0=\landa u^2$, where $\landa$ is the conformal factor
of the metric of $\psi$ and $u$ its angle function,
\item $h_z(z_0)=\vart_0$.
\end{enumerate}
Moreover, $$N=\frac{8\Re \Big(G_z \big( \tau_0
h_{\bar{z}}-4\overline{Q}h_z\big)\Big)}{\tau_0^2-16|Q|^2}  + G\,
\sqrt{\displaystyle\frac{\tau_0+4|h_z|^2}{\tau_0}},$$
 and $h:\Sigma\flecha \R$ is the unique (up to additive constants) solution to the differential system below such that $h_z(z_0)=\vart_0$:
  \begin{equation}\label{eq:A}
\left\{ \begin{array}{lll}
 h_{zz} & = & (\log\tau_0)_z\,  h_z - Q\, \sqrt{\displaystyle\frac{\tau_0+4|h_z|^2}{\tau_0}}, \\
\\
h_{z\bar z} &=&
\displaystyle\frac{1}{4}\sqrt{\tau_0(\tau_0+4|h_z|^2)}.
\end{array}\right.
\end{equation}

\end{enumerate}
\end{teo}

\begin{remark}\label{loqueisa}
A straightforward consequence of the definition of the hyperbolic
Gauss map is that if two surfaces in $\H^2\times\R$ with r.v.p.
are congruent by an isometry preserving the orientation of the
fibers, then their respective hyperbolic Gauss maps differ by an
isometry of $\H^2$. As a result, if two harmonic maps $G_1,G_2$
into $\H^2$ differ by an isometry, then the two families of
$H=1/2$ surfaces obtained from $G_1$ and $G_2$ via Theorem
\ref{th:american} just differ by an isometry of $\H^2\times \R$
preserving the orientation of the fibers. This is a consequence
of the uniqueness part of Theorem \ref{th:american}.
\end{remark}

%%%%%%%%%%%%%%%%%%%%%%%%%%%%%%%%%%%%%%%%%%%%%%%%%%%

\subsection*{The canonical examples in $\Nil$}\label{sub:canonical}

In what follows, $\Sigma$ will denote a simply-connected open
Riemann surface, that will be assumed without loss of generality
to be $\Sigma=\C$ or $\Sigma=\D$.

In \cite{FeMi1} we proved the existence of a large family of
complete minimal surfaces in $\Nil$ that are local vertical
graphs. More specifically, we showed:

\begin{teo}[\cite{FeMi1}]\label{canex}
Let $Q dz^2$ denote a holomorphic quadratic differential on
$\Sigma$, which is non-zero if $\Sigma =\C$. Then there exists
(at least) a complete minimal surface with regular vertical
projection in $\Nil$ whose Abresch-Rosenberg differential is $Q
dz^2$.
\end{teo}

It is relevant for our purposes to explain how these examples are
constructed. Let $Qdz^2$ be a holomorphic quadratic differential
as above. By \cite{Wan,WaAu} there exists a complete spacelike
CMC $1/2$ surface $f:\Sigma\to\L^3$ in the Lorentz-Minkowski
$3$-space $\L^3$ whose Hopf differential is $Qdz^2$. Here, recall
that $\L^3$ is just $\R^3$ endowed with the Lorentzian metric
$\esiz, \esde = dx_1^2 +dx_2^2 -dx_3^2$ in canonical coordinates,
and that an immersed surface $f:\Sigma\flecha \L^3$ is called
\emph{spacelike} if its induced metric $\esiz df,df\esde$ is
Riemannian.

Let $\nu:\Sigma\flecha \H^2 \cup \H_-^2\subset \L^3$ denote its
Gauss map, which is well known to be harmonic. Now let
$G:\Sigma\flecha \H^2$ denote the harmonic map given by $G=\nu$
if $\nu$ is future-pointing, or by the vertical symmetry of $\nu$
in $\L^3$ otherwise. Thus, if $\esiz df , df\esde = \tau_0
|dz|^2$ denotes the induced metric of $f$, it turns out (see
\cite{FeMi1}) that $\{Q,\tau_0\}$ are Weierstrass data for $G$.
Then the $H=1/2$ surfaces in $\H^2\times \R$ constructed from $G$
and $\tau_0$ via Theorem \ref{th:american} are complete (since
$\tau_0 |dz|^2 = \landa u^2 |dz|^2$ is complete), have regular
vertical projection, and Abresch-Rosenberg differential $-Qdz^2$.
Consequently, the sister surfaces of these examples are complete
minimal surfaces in $\Nil$ with regular vertical projection and
whose Abresch-Rosenberg differential is $Q dz^2$.

Let us mention that $f:\Sigma\flecha \L^3$ is unique, in the
sense that if $\tilde{f}:\Sigma\flecha \L^3$ is another complete
spacelike $H=1/2$ surface with Hopf differential $Qdz^2$, then
$\tilde{f} = \Phi \circ f$ for some positive rigid motion $\Phi$
of $\L^3$. Thereby, if we start with some $\tilde{f}$ different
from (but congruent to) $f$, the harmonic map
$\tilde{G}:\Sigma\flecha \H^2$ we obtain following the process
above differs from $G$ at most by an isometry of $\H^2$.

Thus, bearing in mind Remark \ref{loqueisa} we conclude that the
families of $H=1/2$ surfaces in $\H^2\times \R$ obtained,
respectively, from $f$ and $\tilde{f}$, differ at most by an
isometry of $\H^2\times \R$ that preserves the orientation of the
fibers. This means by the sister correspondence that no new
examples in $\Nil$ appear if in the above process we consider
$\Phi \circ f$ instead of $f$ for any positive rigid motion
$\Phi$ of $\L^3$.

Finally, let us comment some properties of these minimal examples
in $\Nil$, that one should keep in mind for what follows.

\begin{itemize}
\item
By Theorem \ref{th:american} and the uniqueness property we have
just pointed out, the above procedure yields in the generic case
(i.e. in the case where the Gauss map $G$ does not have a
$1$-parameter symmetry group) a $2$-parameter family of
non-congruent complete minimal surfaces in $\Nil$ with the same
Abresch-Rosenberg differential. These two continuous parameters
come from the variation of the initial condition $h_z
(z_0)=\vart_0$. Although the elements of this family are
non-congruent in general (e.g. if the Gauss map has no
symmetries), it happens that for minimal surfaces in $\Nil$
invariant under a $1$-parameter subgroup of isometries many of
the elements of the $2$-parameter family are congruent to each
other. A detailed discussion on this topic can be found in page
1165 of \cite{FeMi1} (see also \cite{Dan2}).
 \item
The only minimal surfaces in $\Nil$ with $Q=0$ and $\Sigma =\C$
are vertical planes, which do not have regular vertical
projection. Besides, minimal surfaces with $Q=0$ and $\Sigma =
\D$ are the horizontal minimal umbrellas \cite{AbRo2,Dan1}.
 \item
When $\Sigma=\C$ and $Qdz^2=c\,dz^2$, $c\in\C\setminus\{0\}$ the
spacelike surface $f:\Sigma\to\L^3$ defined above is the
hyperbolic cylinder, whose Gauss map parametrizes a piece of a
geodesic in $\H^2$. This corresponds to Case 1 in Theorem
\ref{th:american}, and therefore, the corresponding family of
canonical examples is the $1$-parameter family of saddle type
surfaces, as stated in the introduction.
 \item
Any of the minimal surfaces in $\Nil$ given by Theorem
\ref{canex} verifies that $u^2 \, g$ is a complete Riemannian
metric, where here $g$ and $u$ stand, respectively, for the
induced metric and the angle function of the minimal surface.
\end{itemize}

\begin{defi}[Canonical Examples]\label{def:canonical}
A minimal surface $X:\Sigma\to\Nil$ will be called a
\emph{canonical example} if it is one of the complete minimal
surfaces with regular vertical projection given by Theorem
\ref{canex}.
\end{defi}

We will show in the next section that the canonical examples in
$\Nil$ are exactly the entire minimal graphs in $\Nil$. This
justifies the terminology we have used.

%%%%%%%%%%%%%%%%%%%%%%%%%%%%%%%%%%%%%%%%%%%%%%%%%%%%%%%%%%%%%%%%%%%%%%%%%%%%%%%%%%%%%%%

\section{Solution to the Bernstein problem in $\Nil$}\label{mainsec}

This section is devoted to the proof of Theorem \ref{th:main}.
First of all, we prove that any canonical example is an entire
graph. This will be obtained as a consequence of the following
Lemma.

\begin{lem}\label{lem:autov} Let $X:\Sigma\to\Ek$ be a surface
in the homogeneous $3$-manifold $\Ek$. Denote by $F:\Sigma\to\Mk$
its vertical projection and by $g_F$ the induced metric of $F$ in
$\Sigma$. Then, for every $v\in \X (\Sigma)$,
 \begin{equation}\label{trescero}
 u^2 \, g(v,v) \leq g_F (v,v)\leq g (v,v) ,
  \end{equation}
where $g$ denotes the induced metric of the immersion $X$ and $u$
is the angle function.
\end{lem}
\begin{proof}
Let $z$ be a conformal parameter for $X$ so that the induced
metric on $\Sigma$ is written as $g=\landa|dz|^2$. Keeping the
notation introduced in Section \ref{sec:prelim} we have by
\eqref{xip} that the metric of the vertical projection is written
with respect to $z$ as
\begin{equation}\label{eq:proyecc}
g_F= -A^2dz^2 + (\landa-2|A|^2)|dz|^2 -\bar{A}^2d\bar{z}^2.
\end{equation}
Then the eigenvalues $\rho_1\geq \rho_2$ of the symmetric
bilinear form $g_F$ with respect to the flat metric $|dz|^2$ are
the solutions of the equation $\det [g_F-\rho\, \mbox{Id}]=0$,
which by {\bf (C.4)} yields $$\rho_1=\landa, \hspace{1cm} \rho_2=
\landa - 4|A|^2 = \landa u^2.$$ From this we obtain
\eqref{trescero} immediately.
\end{proof}

From this Lemma we can now prove that canonical examples are
entire minimal graphs. Indeed, if $X:\Sigma\to\Nil$ is a
canonical example with angle function $u$ and induced metric $g$
on $\Sigma$, the metric $u^2\, g$ is complete, and thus Lemma
\ref{lem:autov} gives the completeness of $g_F$. As a
consequence, $F:(\Sigma,g_F) \flecha \R^2$ is a local isometry
with $(\Sigma, g_F)$ complete. In these conditions, $F$ must
necessarily be a (surjective) covering map. Hence, the vertical
projection $F:\Sigma\to\R^2$ is a global diffeomorphism and $X$
is an entire graph.

This fact together with the comments in Section \ref{sec:prelim}
regarding the canonical examples in $\Nil$ proves the existence
part of Theorem \ref{th:main}.

Once here, in order to end up the proof of Theorem \ref{th:main},
it remains only to check that any entire minimal graph in $\Nil$
is one of the canonical examples. For that, let us first provide
a geometric interpretation of the differential system
\eqref{eq:A}.

From now on, $\Sigma$ will denote an open simply connected
Riemann surface.

Let $f:\Sigma\flecha \L^3$ denote a spacelike surface with
future-pointing Gauss map, with induced metric $\tau_0 |dz|^2$
and Hopf differential $Q dz^2$. Then:
 \begin{enumerate}
 \item
If $H=1/2$, its Gauss map $G:\Sigma\flecha \H^2$ is harmonic and
admits the Weierstrass data $\{Q,\tau_0\}$. In addition,
$Q,\tau_0$ verify the Gauss-Codazzi equations in $\L^3$ for
$H=1/2$, i.e.
 \begin{equation}\label{gcl3}
 (\log \tau_0)_{z\bar{z}} = \frac{\tau_0}{8}- \frac{2|Q|^2}{\tau_0}, \hspace{1.2cm} Q_{\bar{z}} =0.
 \end{equation}
 \item
Conversely, if a harmonic map $G:\Sigma\flecha \H^2$ admits
Weierstrass data $\{Q,\tau_0\}$, then by \cite[Lemma 4]{FeMi1}
the pair $Q,\tau_0$ satisfy \eqref{gcl3}. Therefore there exists
some (unique up to positive isometries) spacelike $H=1/2$ surface
$f:\Sigma\flecha \L^3$ with metric $\tau_0 |dz|^2$ and Hopf
differential $Q dz^2$.
 \end{enumerate}
Bearing these facts in mind, we have

\begin{lem}\label{lem:solution}
Let $\{Q,\tau_0\}$ be Weierstrass data of a harmonic map
$G:\Sigma\flecha \H^2$, and let $f:\Sigma\flecha \L^3$ denote an
arbitrary  spacelike $H=1/2$ surface in $\L^3$ with Hopf
differential $-Q dz^2$ and metric $\tau_0 |dz|^2$ (which is unique
up to positive isometries of $\L^3$).

Then a function $h:\Sigma\flecha \R$ is a solution to
\eqref{eq:A} if and only if $h=-\esiz f,a\esde$, where $a\in
\L^3$ verifies $\esiz a,a\esde=-1$ and belongs to the connected
component of the two-sheeted hyperboloid $\H^2\cup \H_-^2\subset
\L^3$ where the Gauss map of $f$ lies.
\end{lem}
\begin{proof}
Let us first of all prove that $h:=-\esiz f, a\esde$ as above
solves \eqref{eq:A}. For that, we first rewrite \eqref{eq:A} as
 \begin{equation}\label{eq:B}
\left\{ \begin{array}{lll}
 h_{zz} & = & (\log\tau_0)_z\,  h_z - \displaystyle \frac{4Q}{\tau_0} h_{z\bar{z}}, \\
\\
h_{z\bar z} &=&
\displaystyle\frac{1}{4}\sqrt{\tau_0(\tau_0+4|h_z|^2)}.
\end{array}\right.
 \end{equation}
On the other hand, by the structure equations for $f$ in $\L^3$
we have
 \begin{equation}\label{estl3}
 f_{zz} = (\log \tau_0)_z f_z - Q G_f, \hspace{1cm} f_{z\bar{z}} = \frac{\tau_0}{4} G_f,
 \end{equation}
where $G_f :\Sigma\flecha \H^2\cup \H_-^2$ is the Gauss map of
$f$. So, it is clear by \eqref{estl3} that $h=-\esiz f,a\esde$
verifies the first equation in \eqref{eq:B} for every $a\in
\L^3$. To see that $h$ also verifies the second one we first
observe that, by \eqref{estl3},
 \begin{equation}\label{isa2}
 h_{z\bar{z}} = -\frac{\tau_0}{4} \esiz a, G_f\esde.
 \end{equation}
Besides, $$a= \frac{2\esiz a,f_{\bar{z}}\esde}{\tau_0} \, f_z +
\frac{2\esiz a,f_{z}\esde}{\tau_0} \, f_{\bar{z}} - \esiz
a,G_f\esde G_f,$$ so from $\esiz a,a\esde =-1$ we obtain
 \begin{equation}\label{isa3}
 \esiz a,G_f\esde^2 = 1+ \frac{4 |\esiz a,f_z\esde |^2}{\tau_0}.
 \end{equation}
Putting together \eqref{isa2}, \eqref{isa3} and taking into
account that $\esiz a,G_f\esde <0$ by hypothesis, we conclude
that $$h_{z\bar{z}}= \frac{\tau_0}{4} \sqrt{1+ \frac{4 |\esiz
a,f_z\esde |^2}{\tau_0}} =
\displaystyle\frac{1}{4}\sqrt{\tau_0(\tau_0+4|h_z|^2)},$$ i.e.
$h$ satisfies also the second equation of \eqref{eq:B}, and
therefore \eqref{eq:A} as claimed.

Finally, in order to see that all the solutions of \eqref{eq:A}
are of this form, we will see that for every $\vart_0\in \C$
there is some $a\in \H^2\cup \H_-^2$ in the connected component
of $G(\Sigma)$ such that
 \begin{equation}\label{conind}
 -\esiz f_z (z_0),a\esde = \vart_0
 \end{equation}
for a fixed point $z_0\in \Sigma$. Once this is done, the result
will follow by the uniqueness (up to additive constants) of
solutions to \eqref{eq:A} with the initial condition $h_z(z_0)=
\vart_0$.

To prove this claim, observe first that \eqref{conind} is a
linear system of two equations for $a=(a_1,a_2,a_3)$. So, as $df$
is spacelike and regular, the solution to the system above is an
affine timelike line in $\L^3$. This line meets both $\H^2$ and
$\H_-^2$, what ensures the existence of some $a$ with $\esiz
a,a\esde =-1$ in the connected component of $G(\Sigma)$ that
fulfills the initial condition \eqref{conind}.

\end{proof}

The next Lemma is the key ingredient to prove this classification
result. We also think that this result has importance by itself
beyond the use that we give it here. In particular, it may allow
to construct a Sym-Bobenko type formula for $H=1/2$ surfaces in
$\H^2\times \R$ and for minimal surfaces in $\Nil$.

\begin{lem}\label{lem:sister}
Let $\psi =(N,h):\Sigma\to\H^2\times\R$ be a $H=1/2$ surface with
regular vertical projection, and let $-Q dz^2$ denote its
Abresch-Rosenberg differential. Let $X=(F,t):\Sigma\to\Nil$ be
any of its sister minimal surfaces in $\Nil$.

Then there is some $\varepsilon = \pm 1$ such that
$f:=(F,\varepsilon h):\Sigma\to\L^3$ is the unique (up to
positive isometries) spacelike surface with $H=1/2$ in $\L^3$
with induced metric $u^2 g$ and Hopf differential $-Q dz^2$. Here
$g,u$ are, respectively, the metric and the angle function of $X$
(and of $\psi$).
\end{lem}

\begin{proof}
Let $z$ denote a conformal parameter on $\Sigma$, and write
$g=\landa |dz|^2$ for the common induced metric of both $\psi$
and $X$. Let us point out that $Q dz^2$ is the Hopf differential
of the (harmonic) hyperbolic Gauss map $G:\Sigma\flecha \H^2$ of
$\psi$, and that $G$ admits the Weierstrass data $\{Q,\tau_0 :=
\landa u^2\}$. Thus $\{-Q,\tau_0\}$ verify the Gauss-Codazzi equations \eqref{gcl3}, and there exists a unique (up to positive isometries of $\L^3$) spacelike $H=1/2$
immersion $f_0:\Sigma\flecha \L^3$ with induced metric $\tau_0
|dz|^2$ and Hopf differential $-Q dz^2$.

Let us denote $f_0=(P,r):\Sigma\flecha \L^3$. Then by Lemma
\ref{lem:solution} any solution to \eqref{eq:A} is of the form
$-\esiz f_0,a \esde$ for some $a\in \H^2\cup \H_-^2$. In
particular, the height function $h$ of $\psi =(N,h)$ is of this
form, by Theorem \ref{th:american}. Thus, by composing with an
adequate positive isometry of $\L^3$ if necessary, we may assume
that $h=r$.

With this, the metric of the projection $P$ is $$\esiz dP,dP\esde
= h_z^2 dz^2 + \left(\tau_0 + 2|h_z|^2\right) |dz|^2 +
h_{\bar{z}}^2 d\bar{z}^2.$$ On the other hand, let $X=(F,t)$ be
an arbitrary sister minimal surface in $\Nil$ of $\psi =(N,h)$. Then, by \eqref{eq:fund1} and \eqref{eq:fund2} we have that
$A=ih_z$, where $A$ belongs to the fundamental data of $X$. Now,
using this together with  \eqref{eq:proyecc} and formula
{\bf (C.4)} we infer that $$\def\arraystretch{1.8} \begin{array}{lll}
\esiz dF,dF\esde &= &-A^2 dz^2 + (\landa -2 |A|^2) |dz|^2 - \bar{A}^2 d\bar{z}^2 \\
& = & h_z^2 dz^2 + (\landa u^2 +2 |h_z|^2) |dz|^2 +
h_{\bar{z}}^2  d\bar{z}^2 \\ & = & h_z^2 dz^2 + \left(\tau_0 +
2|h_z|^2\right) |dz|^2 + h_{\bar{z}}^2 d\bar{z}^2,
\end{array}$$ i.e. $\esiz dF,dF\esde = \esiz dP,dP\esde$. Then $F=\Phi \circ P$ for some isometry $\Phi:\R^2\flecha \R^2$. Now, this isometry can be extended to a positive rigid motion $\Psi$ of $\L^3$, which may possibly reverse time-orientation. So, we can conclude the existence of some $\ep = \pm 1$ such that $\Psi\circ f_0=(F,\ep h)=:f$. This concludes the proof of Lemma \ref{lem:sister}.
\end{proof}

Once here we can finish the proof of Theorem \ref{th:main}. For
that, let us consider $X=(F,t):\Sigma\to\Nil$ to be an entire
minimal graph with Abresch-Rosenberg differential $Q dz^2$.
First, take its sister $H=1/2$ immersion
$\psi=(N,h):\Sigma\flecha \H^2\times \R$, which has
Abresch-Rosenberg differential $-Q dz^2$, induced metric $\landa
|dz|^2$, angle function $u$ (we shall assume that $u>0$ without
loss of generality), and is defined up to isometries of
$\H^2\times \R$ preserving the orientations of both base and
fibers. Let $G:\Sigma\flecha \H^2$ denote the (harmonic)
hyperbolic Gauss map of one of such sister immersions $\psi$.
Then $X:\Sigma\flecha \Nil$ will be a canonical example if $G$
differs at most by a rigid motion of $\L^3$ from the Gauss map
$\nu:\Sigma\flecha \H^2\cup \H_-^2\subset \L^3$ of some complete
spacelike surface $f_0:\Sigma\flecha \L^3$ with $H=1/2$. This is
just a consequence of the proper construction of the canonical
examples.

But now, by Theorem \ref{th:american}, the hyperbolic Gauss map
$G$ admits the Weierstrass data $\{Q, \landa u^2\}$. This means
that there is a unique (up to positive rigid motions) spacelike
surface $f_0:\Sigma\flecha \L^3$ with $H=1/2$, induced metric
$\landa u^2 |dz|^2$ and Hopf differential $Q dz^2$. By composing
if necessary with a positive isometry of $\L^3$ that reverses
time-orientation we can assume that the Gauss map of $f_0$ is
future-pointing. Moreover, it has the Weierstrass data $\{Q,
\landa u^2\}$. Now by \cite[Lemma 6]{FeMi1} this Gauss map
differs from $G$ at most by a rigid motion.

So, we only need to check that $f_0$ is complete, i.e that
$\landa u^2 |dz|^2$ is complete for $X=(F,t):\Sigma\flecha \Nil$.
By Lemma \ref{lem:sister} we know the existence of some spacelike
$H=1/2$ surface $f=(P,r):\Sigma\flecha \L^3$ such that $P=F$.
Thus, $f$ is an entire spacelike graph in $\L^3$. Now, a result
in \cite{ChYa} establishes that all spacelike entire CMC graphs
in $\L^3$ have complete induced metric (here the CMC assumption
is essential). In particular $f$ is complete. But again by Lemma
\ref{lem:sister}, the induced metric of $f$ is $\landa u^2
|dz|^2$. Thus, this metric is complete and $X$ must be a
canonical example.

We have then proved:

\begin{teo}\label{th:can}
A minimal surface in $\Nil$ is an entire vertical graph if and
only if it is a canonical example.
\end{teo}

From this result and the description of the canonical examples in
Section \ref{sec:prelim} we obtain Theorem \ref{th:main}.
Moreover, as a byproduct of our proof we get the following
characterization of the entire minimal graphs in $\Nil$.

\begin{cor}\label{corug} Let $X$ be a minimal surface in $\Nil$ with metric $g$
and angle function $u$. Then $X$ is an entire graph if and only
if $u^2\, g$ is a complete Riemannian metric.
\end{cor}

Observe that if $u^2 g$ is a complete metric for an arbitrary
surface in any homogeneous space $\Ek$ (not necessarily with
CMC), then the surface must be an entire graph by Lemma
\ref{lem:autov}. However, the converse is not true in general.
For instance, the entire graph $$x_3(x_1,x_2)= x_2 \,
e^{1/(1-x_1^2-x_2^2)}, \hspace{1cm} x_1^2 + x_2^2 <1,$$ in
$\H^2\times \R \equiv (\D,ds_P^2)\times\R$, has non-complete $u^2
g$ metric. Indeed, it is an elementary computation to check that
the divergent curve $\alpha(t):= (t,0,x_3(t,0))$, $0\leq t <1$,
lies on the graph and has finite $u^2 g$ length.

%%%%%%%%%%%%%%%%%%%%%%%%%%%%%%%%%%%%%%%%%%%%%%%%%%%%%%%%%%%%%%%%%%%%%%%%%%%%%%%%%%%%%%%%%%%%%%%%%%%%%%%%%%%%%%%%%
\section{Surfaces with $H=1/2$ in $\H^2\times \R$ and further
results}\label{sec:H2xR}

In view of the results of the present paper and Daniel's
Lawson-type correspondence, it is a natural problem to ask for
the classification of the entire vertical graphs with $H=1/2$ in
the product space $\H^2\times \R$ (see \cite{FeMi1,Sa}). The
following result, which constitutes an improvement of
\cite[Theorem 16]{FeMi1}, shows that the class of solutions to
such Bernstein-type problem in $\H^2\times \R$ is at least as
large as the class of entire minimal graphs in $\Nil$.

\begin{pro}\label{proH2}
The sister surface of an entire minimal graph in $\Nil$ is an
entire $H=1/2$ graph in $\H^2\times\R$.

Consequently, given a holomorphic quadratic differential $Q dz^2$
on $\Sigma \equiv \C$ or $\D$ which is not identically zero if $\Sigma\equiv \C$, there exists a
$2$-parametric family of (generically non-congruent) entire
$H=1/2$ graphs in $\H^2\times \R$ whose Abresch-Rosenberg
differential is $Qdz^2$.
\end{pro}
\begin{proof}
It follows directly from Corollary \ref{corug}, Lemma
\ref{lem:autov} and the fact that the metric $\landa u^2 |dz|^2$
 is invariant by the sister correspondence.
\end{proof}

\begin{remark}
By Corollary \ref{corug}, the entire $H=1/2$ graphs in Proposition
\ref{proH2} are characterized by the fact that $u^2 g$ is a
complete metric. It is natural to ask if these are the only
entire $H=1/2$ graphs in $\H^2\times \R$. Unfortunately, we
cannot proceed in this setting as we did for minimal graphs in
$\Nil$, since we do not have a good control on the vertical
projection for $H=1/2$ surfaces in $\H^2\times \R$. Let us also
point out that the existence of an entire $H=1/2$ graph in
$\H^2\times \R$ for which  $u^2 g$ is not complete remains unknown.
\end{remark}

Another interesting application of the arguments used in the
proof of Theorem \ref{th:main} is the following.

\begin{cor}\label{crG}
Let $\psi:\Sigma\flecha \H^2\times \R$ denote a $H=1/2$ surface
with regular vertical projection such that its hyperbolic Gauss
map $G:\Sigma\flecha \H^2$ is a global diffeomorphism. Assume
additionally that its Abresch-Rosenberg differential vanishes
somewhere. Then $\psi$ is one of the entire graphs given in
Proposition \ref{proH2}.
\end{cor}
\begin{proof}
By Theorem \ref{th:american}, $G$ has Weierstrass data
$\{-Q,\landa u^2\}$, where $Q dz^2$ is the Abresch-Rosenberg
differential of $\psi$. Thus $$\esiz dG,dG\esde = -Q dz^2 +
\left(\frac{\landa u^2}{4} + \frac{4 |Q|^2}{\landa u^2}\right)
|dz|^2 - \bar{Q} d\bar{z}^2.$$ If we now compute the eigenvalues
$\rho_1\geq \rho_2$ of $\esiz dG,dG\esde$ with respect to the
flat metric $|dz|^2$, we obtain $$\rho_1 = \frac{\landa u^2}{4} (
1 + d_G)^2, \hspace{1cm} \rho_2 = \frac{\landa u^2}{4} ( 1 -
d_G)^2, \hspace{0.5cm} d_G:= \frac{4|Q|}{\landa u^2}.$$ Hence, as
$\esiz dG,dG\esde$ is complete by hypothesis, the metric $\rho_1
|dz|^2$ is complete. But since $\esiz dG,dG\esde$ is regular, we
infer that $4|Q|\neq \landa u^2$ at every point. So, as $Q$
vanishes somewhere we get $d_G \in [0,1)$. By this fact and the
completeness of $\rho_1 |dz|^2$ we conclude that $\landa u^2
|dz|^2$ is a complete metric, and hence by Corollary \ref{corug}
that $\psi$ is one of the entire graphs of Proposition
\ref{proH2}.
\end{proof}

\begin{remark}
If $G:\Sigma\flecha \H^2$ admits Weierstrass data $\{Q,\tau_0\}$
such that $Q$ never vanishes, then there is a second set of
Weierstrass data $\{Q, \tau^{\sharp}:= 16|Q|^2 / \tau_0\}$
available for $G$. These two sets of data give rise to
\emph{parallel-like} $H=1/2$ surfaces $\psi,\psi^{\sharp}$ in
$\H^2\times \R$ with the same hyperbolic Gauss map $G$ (see
\cite[Theorem 19]{FeMi1}). In these conditions, we can still say
that if $G$ is a global diffeomorphism, then one of $\psi,
\psi^{\sharp}$ is an entire graph of Proposition \ref{proH2}. This
follows by a simple modification of the last part of the above
proof.
\end{remark}

\begin{remark}
The existence of harmonic diffeomorphisms from $\C$ onto $\H^2$
has been recently proved by Collin and Rosenberg in \cite{CoRo}.
\end{remark}

It is interesting to observe that Corollary \ref{crG} has a
counterpart for minimal surfaces in $\Nil$. Indeed, in
\cite{Dan2} Daniel gave an extension to the case of minimal
surfaces in $\Nil$ of the harmonicity theory for the hyperbolic
Gauss map of $H=1/2$ surfaces in $\H^2\times \R$. In this $\Nil$
setting, the harmonic Gauss map is simply the usual Gauss map of
a surface with regular vertical projection (taking values on the
upper hemisphere of the Lie algebra of $\Nil$), composed with
stereographic projection. This gives a Gauss map $g:\Sigma\flecha
\D$ which is harmonic for minimal surfaces when we endow $\D$
with the Poincaré metric. See \cite{Dan2} for more details.

Taking these facts into consideration, one can see that the proof
of Corollary \ref{crG} also works in $\Nil$, and so we obtain:

\begin{cor}
Let $X:\Sigma\flecha \Nil$ denote a minimal surface with regular
vertical projection, and whose Abresch-Rosenberg differential
vanishes somewhere. Assume that its usual Gauss map is a global
diffeomorphism onto the upper hemisphere $\S_+^2$. Then $X$ is
canonical example.
\end{cor}

We end up with two closing remarks regarding our results. The
study of the Bernstein problem in $\Nil$ presented here makes
constant use of the theory of $H=1/2$ surfaces in $\H^2\times
\R$, for instance, for the construction of the canonical
examples. It is nonetheless very likely that our arguments can be
modified in certain aspects to give an alternative proof of
Theorem \ref{th:main} without using $\H^2\times \R$ at all. This
possibility is given by the harmonicity of the Gauss map for
minimal surfaces in $\Nil$ \cite{Dan2}. We have followed here the
path across $\H^2\times \R$ to give a natural continuation of the
work \cite{FeMi1}, which is where the canonical examples were
first constructed, and also because we believe that this more
flexible perspective may give a clearer global picture of CMC
surface theory in homogeneous $3$-manifolds. Lemma
\ref{lem:sister} is an example of the benefits that this
perspective can have.

Let us finally point out that there are still relevant problems
regarding the global geometry of entire minimal graphs in $\Nil$.
For instance, the problem of describing the global behavior of
canonical examples in $\Nil$  is specially interesting.  Another
interesting inquiry line is to give concrete geometric conditions
that characterize the horizontal umbrellas or the saddle-type
minimal graphs among all canonical examples (let us point out
that these minimal surfaces are the only examples invariant by a
$1$-parameter isometry subgroup).

\subsubsection*{Added after revision}

Very recently, L. Hauswirth, H. Rosenberg and J. Spruck
\cite{hrs} have proved that any complete local graph with mean
curvature one half in $\H^2\times\R$ must be an entire graph.
After that work, the analogous result for minimal
surfaces in $\Nil$ has been proved by B. Daniel and L. Hauswirth
\cite{dh}.

As a consequence, it was obtained in \cite{dh} that any entire graph with $H=1/2$ in $\H^2\times\R$ must be the sister surface of a canonical example in $\Nil$. This result gives a classification of the entire $H=1/2$ graphs in $\H^2\times \R$, proving that any such entire graph is one of the examples given in Proposition \ref{proH2}.

\def\refname{References}


\begin{thebibliography}{9}

\bibitem[AbRo1]{AbRo1} U. Abresch, H. Rosenberg, A Hopf differential for constant mean curvature surfaces in $\S^2\times \R$ and $\H^2\times \R$, {\it Acta Math.} {\bf 193} (2004), 141--174.

\bibitem[AbRo2]{AbRo2} U. Abresch, H. Rosenberg, Generalized Hopf differentials, {\it Mat. Contemp.} {\bf 28} (2005), 1--28.

\bibitem[ADR]{ADR} L.J. Alías, M. Dajczer, H. Rosenberg, The Dirichlet problem for CMC surfaces in Heisenberg space,
{\it Calc. Var. Partial Diff. Equations}, {\bf 30} (2007),
513--522.

\bibitem[BSV]{BSV} V. Barone-Adesi, F. Serra-Cassano, D. Vittone, The Bernstein problem for intrinsic graphs in the Heisenberg group and calibrations,
  {\it Calc. Var. Partial Diff. Equations}, {\bf 30} (2007), 17--49.

\bibitem[Bry]{Bry} R.L. Bryant, Surfaces of mean curvature one in
hyperbolic space, {\it Astérisque}, {\bf 154-155} (1987), 321--347.

\bibitem[ChYa]{ChYa} S.Y. Cheng, S.T. Yau, Maximal spacelike hypersurfaces in the
Lorentz-Minkowski spaces, {\it Ann. of Math.} {\bf 104} (1976),
407--419.

\bibitem[CHMY]{CHMY} J.H. Cheng, J.F Hwang, A. Malchiodi, P. Yang, Minimal surfaces in pseudohermitian geometry, {\it Ann. Sc. Norm. Super. Pisa Cl. Sci.} {\bf 4} (2005), 129--177.

\bibitem[CoRo]{CoRo} P. Collin, H. Rosenberg, {Construction of harmonic diffeomorphisms and minimal graphs}, preprint, 2007.

\bibitem[Dan1]{Dan1} B. Daniel, Isometric immersions into $3$-dimensional homogeneous manifolds, {\it Comment. Math. Helv.} {\bf 82} (2007), 87-131.

\bibitem[Dan2]{Dan2} B. Daniel, The Gauss map of minimal surfaces in the Heisenberg group, preprint, 2006.

\bibitem[DaHa]{dh} B. Daniel, L. Hauswirth, {Half-space theorem,
embedded minimal annuli and minimal graphs in the Heisenberg
space,} preprint, 2007.

\bibitem[FeMi1]{FeMi1} I. Fernández, P. Mira, Harmonic maps and constant mean curvature surfaces in $\H^2\times \R$,
{\it Amer. J. Math.} {\bf 129} (2007), 1145--1181.

\bibitem[FeMi2]{FeMi2} I. Fernández, P. Mira, A characterization of constant mean curvature surfaces
in homogeneous $3$-manifolds, {\it Diff. Geom. Appl.}, {\bf 25} (2007), 281--289.

\bibitem[FMP]{FMP} C. Figueroa, F. Mercuri, R. Pedrosa, Invariant surfaces of the Heisenberg groups, {\it Ann. Mat. Pura Appl.} {\bf 177} (1999), 173--194.

\bibitem[GaMi]{GaMi} J.A. Gálvez, P. Mira, The Cauchy problem for the Liouville equation
and Bryant surfaces, {\it Adv. Math.}, {\bf 195} (2005), 456--490.

\bibitem[GaPa]{GaPa} N. Garofalo, S. Pauls, { The Bernstein problem in the Heisenberg
group}, preprint, 2003.

\bibitem[HRS]{hrs} L. Hauswirth, H. Rosenberg, J. Spruck { On complete mean curvature $1/2$ surfaces in $\H^2\times\R$},
preprint, 2007.

\bibitem[HsHs]{HsHs} W.Y. Hsiang, W.T. Hsiang, On the uniqueness of isoperimetric solutions and imbedded soap bubbles in noncompact symmetric spaces I, {\it Invent. Math.} {\bf 98} (1989), 39--58.

\bibitem[IKOS]{Iea} J. Inoguchi, T. Kumamoto, N. Ohsugi, Y. Suyama, Differential geometry of curves and surfaces in $3$-dimensional homogeneous spaces II, {\it Fukuoka Univ. Sci. Reports} {\bf 30} (2000), 17--47.

\bibitem[MMP]{MMP} F. Mercuri, S. Montaldo, P. Piu, A Weierstrass representation formula for minimal surfaces in $\H_3$ and $\H^2\times \R$, {\it Acta Math. Sinica}, {\bf 22} (2006), 1603--1612.

\bibitem[RiRo]{RiRo} M. Ritoré, C. Rosales, \emph{Area stationary surfaces in the Heisenberg group $\H^1$}, preprint, 2005.

\bibitem[Sa]{Sa} R. Sa Earp, Parabolic and hyperbolic screw
motion surfaces in $\H^2\times\R$, to appear in {\it J. Austr.
Math. Soc.}, 2007.


\bibitem[Wan]{Wan} T.Y. Wan, Constant mean curvature surface harmonic map and universal Teichmuller space, {\it J. Differential Geom.} {\bf 35} (1992), 643--657.

\bibitem[WaAu]{WaAu} T.Y. Wan, T.K. Au, Parabolic constant mean curvature spacelike surfaces, {\it Proc. Amer. Math. Soc.} {\bf 120} (1994), 559-564.

\end{thebibliography}
\end{document}